\documentclass[11pt,twoside]{amsart}
\usepackage{amsmath, amsthm, amscd, amsfonts, amssymb, graphicx, color}
\usepackage[bookmarksnumbered, colorlinks, plainpages]{hyperref}

\newtheorem{thm}{Theorem}[section]

\newtheorem{lem}[thm]{Lemma}

\newtheorem{hypo}[thm]{Hypothesis}

\newtheorem{res}[thm]{Result}
\numberwithin{equation}{section}
\def\pn{\par\noindent}

\newcommand{\Lt}{\mathcal{L}}
\newcommand{\Lnt}{\mathcal{L}_n}
\newcommand{\gnt}{\mathcal{G}_n}
\newcommand{\gntij}{\mathcal{G}^{ij}_n}

\newcommand{\Dtauk}{D_{\tau}^k}
\newcommand{\Drs}{D_{r}X_{s}}

\newcommand{\Drti}{D_{r}X^i_{t}}

\newcommand{\Dnrti}{D_{r}(X^n_{t})^i}
\newcommand{\Dnrsi}{D_{r}(X^n_{s})^i}

\newcommand{\Drsj}{D^j_{r}X_{s}}
\newcommand{\Dnrsj}{D^j_{r}X^{n}_{s}}

\newcommand{\Dnt}{DX^{n}_{t}}

\newcommand{\Dnrt}{D_{r}X^n_{t}}
\newcommand{\Dnrs}{D_{r}X^n_{s}}

\newcommand{\DD}{D^{j,k}_{r,\tau}}

\newcommand{\bpsi}{\nabla b^i(X_{s})}

\newcommand{\bnps}{\nabla b_{n}(X^{n}_{s})}
\newcommand{\bnpsi}{\nabla b^i_{n}(X^{n}_{s})}

\newcommand{\fnpsi}{\nabla f^i(X^{n}_{s})}

\newcommand{\fpsi}{\nabla f^i(X_{s})}

\newcommand{\R}{\mathbb{R}}
\newcommand{\Q}{\mathbb{Q}}
\newcommand{\E}{\mathbb{E}}
\newcommand{\D}{\mathbb{D}}

\begin{document}
%

\title{Weak differentiability of Solutions to SDEs with Semi-Monotone Drifts}
\author{Mahdieh Tahmasebi$^*$ and Shiva Zamani}

\thanks{{\scriptsize
\hskip -0.4 true cm MSC(2010): Primary: 60H07; Secondary: 60H10.
$*$Corresponding author
}}

\maketitle

\begin{abstract}
In this work we prove Malliavin differentiability for the solution to an SDE with locally Lipschitz and semi-monotone drift.  
To this end we construct a sequence of SDEs with globally Lipschitz drifts. We show that the solutions of these SDEs converge to the solution of the original SDE and the $p$-moments of their Malliavin derivatives are uniformly bounded.
\end{abstract}
\vskip 0.2 true cm


\pagestyle{myheadings}
\markboth{\rightline {\scriptsize  Tahmasebi and Zamani}}
         {\leftline{\scriptsize Weak differentiability of Solutions to SDEs}}

\bigskip
\bigskip

\section{Introduction}
\vskip 0.4 true cm
In recent years, there were attemps to generalize existence, uniqueness, and smoothness results  to SDEs with non-globally Lipschitz coefficients, which  have many applications in Financial Mathematics \cite{Higa04, Bavou06, Alos08, Marco1}. 
In \cite{Flandoli10, Zhang10a} the authors studied the existence of a global stochastic flow for SDEs with unbounded and H\"older continuous drifts and  nondegenerate diffusion coefficients. Zhang considered the flow of stochastic transport equations which could have irregular coefficients \cite{Zhang10b}.  \\
The SDE we consider has both nonglobally Lipschitz and 
semi-monotone drift. Such equations  come mostly from finance, biology, and  dynamical systems and are more challenging when considered on infinite dimensional spaces. (see e.g. \cite{B99, Z95, GM07})\\
In this paper, we consider an SDE with locally Lipschitz and monotone drift and globally Lipschitz diffusion. We prove the existence of a unique infinitely Malliavin differentiable strong solution to this SDE. \\
Since the drift of the SDE we consider is not globally Lipschitz, we will construct a sequence of SDEs with globally Lipschitz drifts whose solutions are Malliavin differentiable of all order.  In this way we can apply the classical Malliavin calculus to these solutions. Then we can find  a uniform bound for the moments of all the Malliavin derivatives of solutions. We will prove that the solutions to the constructed sequence of SDEs converge to the solution of the desired SDE. Then by the uniform boundedness of the moments of the mentioned solutions and the convergence result we are able to prove  infinite Malliavin differentiability of  the solution to the original SDE. \\  
The organization of the paper is as follows. In section 2, we recall some basic results from Malliavin calculus that will be used in the paper,  the prerequisites could be found in Nualart's book \cite{Nualart06},  in this section we state also our assumptions and  main results. Section 3  involves the construction of our approximating SDEs with globally Lipschitz coefficients, and the proof of convergence of their solutions to the unique solution of the original SDE (\ref{equa}). In section 4, we will prove uniform boundedness of the Malliavin derivatives associated to the approximating processes, which results to the infinitely weak differentiability of the solution to  SDE  (\ref{equa}). 
\section{Some basic results from Malliavin calculus}\label{ma}
\vskip 0.4 true cm
Let $\Omega$ denote the Wiener space
$C_{_{0}}([0,T];R^d)$. We  furnish $\Omega$ with the
$\parallel.\parallel_{_{\infty}}$-norm making it a (separable)
Banach space. Consider $(\Omega,\mathcal{F},P)$ a complete
probability space, in which $\mathcal{F}$ is generated by the open
sets of the Banach space, $W_{t}$ is a d-dimensional Brownian motion, and $\mathcal{F}_{t}$ is the filtration generated by $W_t$. \\
Consider the Hilbert space $H:=L^2([0,T];\R^d)$.
Let $\{W(h), h \in H\}$ denote a Gaussian
process associated to the Hilbert space $H$ and
$W(h)=\int_{0}^{\infty}{h(t)dW_t}$.
We denote by $C_{_{p}}^\infty(\R^n)$ the set of all infinitely
continuously differentiable functions $f:R^n \longrightarrow R$ such
that $f$ and all of its partial derivatives have polynomial growth.
Let $\mathcal{S}$ denote the class of all smooth random variables $F:\Omega \longrightarrow \R$ such that 
$F=f(W(h_{_{1}}),...,W(h_{_{n}})),$
for some $f$ belonging to $C_{_{p}}^\infty(\R^n)$ and $h_{_{1}},...,h_{_{n}} \in H$
for some $n \geq 1$.\\
 The derivative of the smooth random variable
$F \in \mathcal{S}$ is an $H$-valued random variable given by
\begin{equation*}
D_t F=\Sigma_{i=1}^n
{\partial_{_{i}}}f(W(h_{_{1}}),...,W(h_{_{n}}))h_{_{i}}(t).
\end{equation*}
The operator  
$D$  from $L^p(\Omega)$ to $L^p(\Omega, H)$ is closable. For every $p \geq 1$, we denote its domain by ${\D}^{1,p}$ which is exactly the
closure of $\mathcal{S}$ with respect to
$\parallel.\parallel_{_{1,p}}$ where
\begin{equation*} \qquad \parallel F\parallel_{_{1,p}}
                         =\Big[E\vert F\vert^p+\parallel
                                DF\parallel_{L^p(\Omega; H)}^p\Big]^\frac{1}{p}.
\end{equation*}
(see \cite{Nualart06}). 
One can also define the $k$-th order derivative of $F$ as a random vector in $[0,T]^k \times \Omega$. We denote by $\D^{k,p}$  the completion of  $\mathcal{S}$ with respect to the norm
\begin{equation*} 
\parallel F\parallel_{_{k,p}}
                         =\Big[E\vert F\vert^p+\parallel
                                D^{i_1, \cdots, i_k} F\parallel_{L^p(\Omega; H^{\otimes k})}^p\Big]^\frac{1}{p},
\end{equation*}
and define $\D^{\infty}:= \bigcap_{k,p} \D^{k,p}$.\\
Consider the following stochastic differential equation\\
\begin{equation}\label{equa}
dX_{t}=[b(X_{t})+f(X_{t})] dt+\sigma(X_{t})dW_{t},  \qquad X_0=x_0.
\end{equation}
where  $b,f:\R^d \longrightarrow \R^d$ are measurable functions and  $\sigma:\R^d \longrightarrow M_{d\times d}(\R)$ is a measurable $C^\infty$ function. 
We denote by $\Lt$ the second-order differential operator associated to SDE (\ref{equa}):
\begin{equation*}
\Lt= \frac{1}{2} \sum_{i,j =1}^{d} (\sigma \sigma^*)_{j}^i(x) \partial_{i} \partial_j +  \sum_{i=1}^d [b^i(x)+f^i(x)]\partial_i
\end{equation*}
where $*$ denotes transpose. We use the upper index to show a specified row, and the subindex to show a specified column of a matrix.\\
Kusouka and Stroock  has shown the following result \cite[Theorem 1.9.]{Kusuoka85}.
\begin{res}\label{solindinf}
Assume that the coefficients $b$, $\sigma$ and $f$ in {\rm (\ref{equa})} are globally Lipschitz and
all of their derivatives have polynomial growth, then {\rm (\ref{equa})} has a strong solution in
$\D^{\infty}$ whose Malliavin derivative satisfies the following linear equation. For every $r  \leq t$\\

$\begin{aligned}
D_rX_t^i  &= \sigma^i(X_r)+\int_{r}^{t}(\nabla b^i(X_s)+\nabla f^i(X_s)) D_rX_s ds \\
& + \int_{r}^{t} \nabla \sigma_l^i (Y_s)D_rX_s dW^l_s.
\end{aligned}$\\
where for  $r  > t$, $D_rX_t=0$.
Also it holds
\begin{equation*}
\sup_{0 \leq r \leq T}\E[ \sup_{r \leq s \leq T}\vert D_r^jX^i \vert] < \infty.
\end{equation*}
\end{res}
 Throughout the paper we assume that $b$, $f$ and $\sigma$ satisfy the following Hypothesis.
\begin{hypo}\label{hypothesis}

\begin{enumerate}
 \item  The function $b$ is an $C^\infty$ uniformly monotone function, i.e., there exists a constant $K>0$ such that 
 for every $x,y \in \R^d$,
\begin{equation}\label{monoton}
<b(y)-b(x),y-x> \leq -K{\vert y-x \vert}^2.
\end{equation}
 where $\langle . , . \rangle$ denotes the scalar product in $\R^d$. Furthermore, $b$ is locally Lipschitz and all of its derivatives have polynomial growth.  i.e.,
 for each $x \in \R^d$ and each multiindex $\alpha$ with $\vert \alpha \vert = m$, there exist positive constants $\gamma_m$ and $q_m$ such that
 \begin{equation}\label{roshdb}
 \vert \partial_\alpha b(x) \vert^2 \leq \gamma_m(1+\vert x \vert^{q_m})
 \end{equation}
 Also, set $\xi:= \max_{m \geq 1} q_m < \infty$.
 \item  The functions $f$ and $\sigma$ are $C^\infty$, globally Lipschitz with Lipschitz constant $k_1$ and all of their derivatives  are bounded. Furthermore $f$ has linear growth, i.e. for every $x \in \R^d$,
 \begin{equation}\label{linear}
 \vert f(x) \vert \leq k_1(1+\vert x \vert).
 \end{equation}
\end{enumerate}
\end{hypo}
Hypothesis (\ref{hypothesis}) yields to the following useful inequalities 
\begin{equation}\label{Bound}
\langle b(a)+f(a) ,a \rangle \vee \vert \sigma(a)\vert^2 \leq \alpha + \beta \vert a \vert^2   \qquad  \forall a \in \R^d,
\end{equation}
where
\begin{equation}\label{alpha-betta}
\alpha := \frac12 \vert b(0)\vert^2 +k_1^2 \vee 2 \vert \sigma(0)\vert ^2,
\quad \textmd{and} \quad \beta := ( -K + 1+ k_1^2) \vee 2k_1^2,
\end{equation}
and 
\begin{equation}\label{ggg}
\langle \nabla b(x)y, y \rangle \leq -K \vert y \vert^2  \qquad \forall x,y \in \R^d.
\end{equation}
It is well-known that by inequality (\ref{Bound}), the SDE (\ref{equa}) has a strong solution $\{X_t\}$ (see e.g., \cite{Mao97} and \cite{Mao12}). The uniqueness of the solution is obtained  using It\^o's formula and Gronwall's inequality (Lemma \ref{xlp}). We will show that this solution is in $\D^{\infty}$. To this end, we first show that $X_t \in L^P(\Omega)$,  does not  explode in finite time, and then we construct an almost everywhere convergent sequence of processes $X^n_t$ whose limit is $X_t$, where all the Malliavin derivatives of $X^n_t$ are uniformly bounded  with respect to $n$.
\section{Approximation of the solution}
\vskip 0.4 true cm
For each $n \geq 1$, define the stopping time $\tau_{n}$ via
$$\tau_{n}:=inf\{t|~ \vert X_{t}\vert \geq n^\xi \}.$$
\begin{lem}\label{xlp}
For each $t \in [0,T]$ and $p > 1$, the unique solution $X_{t}$ of (\ref{equa})
 belongs to $L^p(\Omega)$ and does not explode in finite time.
\end{lem}
\begin{proof}
To proceed, first we use Fatou's lemma to show that $X_t \in L^p(\Omega)$ and does not explode. Then, we prove the uniqueness of the solution to SDE (\ref{equa}). \\
By the definition of $\Lt$ and (\ref{Bound}), we have
\begin{align}
\Lt \vert X_t \vert^p & = p\vert X_t \vert^{p-2} \langle X_t  , b(X_t)+f(X_t) \rangle + \frac{p}{2} \vert X_t \vert^{p-2}\vert \sigma(X_t) \vert^2 \nonumber\\
&+ \frac{p(p-2)}{2} \vert X_t \vert^{p-4}\vert \langle X_t, \sigma(X_t) \rangle \vert^2 \nonumber\\
& \leq  p\vert X_t \vert^{p-2} \langle X_t  , b(X_t)+f(X_t) \rangle + \frac{p(p-1)}{2} \vert X_t \vert^{p-2}\vert \sigma(X_t) \vert^2 \nonumber\\
   & \leq p\Big(\beta+(p-1)k_1^2\Big) \vert X_t \vert^{p} + p\Big(\alpha+(p-1)k_1^2\Big) \vert X_t \vert^{p-2} \nonumber\\
   & =: \beta_p  \vert X_t \vert^{p}+\alpha_p  \vert X_t \vert^{p-2}.\label{Lt}
\end{align}
Applying It\^o's formula and using (\ref{Lt}),
\begin{equation}\label{induc}
\frac{d}{dt} \E\Big[\vert X_{t \wedge \tau_n} \vert^p \Big] = \E\Big[ \Lt \vert X_{t \wedge \tau_n}\vert^p \Big]  \leq  \beta_p \E\Big[\vert X_{t \wedge \tau_n} \vert^{p} \Big] + \alpha_p  \E\Big[ \vert X_{t \wedge \tau_n} \vert^{p-2} \Big] .
\end{equation}
Setting $p=2$ and using Gronwall's inequality, we have
\begin{equation}\label{xtwedge}
\E\Big[\vert X_{t \wedge \tau_n} \vert^2 \Big] \leq  \vert x_0 \vert^2 \alpha_2 exp\{ \beta_2 T\}.
\end{equation}
From (\ref{xtwedge}) we can deduce  the following inequality
\begin{equation*}
(\frac{n}{2}-1)^{\frac{1}{q_0}} P\Big(t \geq \tau_n\Big)  \leq  \vert x_0 \vert^2 \alpha_2 exp\{ \beta_2 T\}.
\end{equation*}
Letting $n$ tend to $\infty$, then $lim_{n \rightarrow \infty}\tau_{n}=\infty$ almost surely, which implies that $X_t$ does not explode in any finite time interval $[0,T]$.
Also, let $n$ tend to infinity in (\ref{xtwedge}) and use Fatou's lemma, then
\begin{equation*}
\E(\vert X_t \vert^2)  \leq \E\big(\displaystyle \liminf_{n \rightarrow \infty} \vert X_{t \wedge \tau_{n}} \vert^2\big)
                        \leq  \displaystyle \liminf_{n \rightarrow \infty} \E\big(\vert X_{t \wedge \tau_{n}} \vert^2\big)
                        \leq  \vert x_0 \vert^2 \alpha_2 exp\{ \beta_2 T\}.
\end{equation*}
Finally by (\ref{induc}) and induction on $p$
 we conclude that $X_{t}\in L^p(\Omega)$.\\
 To prove uniqueness, we assume that the SDE (\ref{equa}) has two strong solutions $X_t$ and $Y_t$. Since $X_t, Y_t \in L^2(\Omega)$, applying It\^o's formula we have  \\

$\begin{aligned}
\frac{d}{dt} \E\Big[\vert X_t - Y_t \vert^2 \Big] &= 2\E\Big[ \langle X_t - Y_t , b(X_t)-b(Y_t) \rangle\Big]\\
& + 2\E\Big[ \langle X_t - Y_t , f(X_t)-f(Y_t) \rangle \Big] \\
&+\E\Big[\vert \sigma(X_t)-\sigma(Y_t) \vert^2   \Big]
\end{aligned}$\\
From which by (\ref{monoton}) and the Lipschitz property of $\sigma$ and $f$ we derive
\begin{equation*}
\frac{d}{dt} \E\Big[\vert X_t - Y_t \vert^2 \Big] \leq (-2K+2k_1) \E\Big[ \vert X_t - Y_t \vert^2 \Big].
\end{equation*}
By Gronwall's inequality which is proved in \cite[Lemma 1.1]{Hasminskii80} we conclude that $\E\Big[\vert X_t - Y_t \vert^2 \Big] =0$. So that
\begin{equation*}
P\Big(  \vert X_t - Y_t \vert =0 \quad  \textmd{for all} ~t \in \Q \cap [0,T]  \Big)=0,
\end{equation*}
where $\Q$ denotes the set of rational numbers. Since $t \longrightarrow \vert X_t -Y_t \vert $ is continuous, then 
 \begin{equation*}
P\Big(  \vert X_t - Y_t \vert =0   \quad  \textmd{for all}~ t \in  [0,T]  \Big)=0,
\end{equation*}
and  the uniqueness is proved.
\end{proof}
For every integer $n >0$ let us choose some smooth functions $\phi_{n}:\R \rightarrow \R$ such that $\phi_{n}=1$ on $A_{n}:=\{x \in \R;~\mid
x\mid \leq n^{\xi} \}$ and $\phi_{n}=0$ outside  $A_{2n^{\xi}}$ ($\xi$ defined in Hypothesis \ref{hypothesis} part (1))  and for each multiindex $L$ with $\vert L \vert =l\geq 1$,
\begin{equation}\label{supderbphi}
\sup_{_{n,x}}\Big( \vert \partial_{_{L}} \phi_n \vert +\vert \langle b , \partial_{L} \phi_{n} \rangle \vert \Big)\leq M_{l}
\end{equation}
 for some $M_{l} > 0$. (See Appendix and the proof of Lemma 2.1.1 in \cite{Nualart06}).
  Now, set
$$b_{n}(x):=\phi_{n}(x)b(x)$$ for every $x \in \R^d$ and  $n >0$. Then $b_{n}$ would be globally Lipschitz and continuously differentiable. By (\ref{roshdb}) for each $x \in \R^d$ and each multiindex $L$ with $\vert L \vert = l$, there exist positive constants $\Gamma_l$ and $p_l$ such that
 \begin{equation}\label{roshdbn}
 \vert \partial_L b_n(x) \vert^2 \leq \Gamma_l(1+\vert x \vert^{p_l}).
 \end{equation}
Now by Result \ref{solindinf}, the SDE (\ref{equan}) has a strong solution  in $\D^{\infty}$, that is, there exists $X_t^n$ in $\D^{\infty}$ which satisfies 
\begin{equation}\label{equan}
X_{t}^n=x_{0}+\int_{0}^{t} [b_{n}(X_{s}^n)+f(X_{s}^n)]ds+\int_{0}^{t}
\sigma(X_{s}^n)dW_{s}
\end{equation}
We   denote by $\Lnt$  the infinitesimal operator associated to  SDE (\ref{equan}):
\begin{equation*}
\Lnt= \frac{1}{2} \sum_{i,j =1}^{d} (\sigma \sigma^*)^i_{_{j}}(x) \partial_{i} \partial_j+\sum_{i=1}^d [b^i_n(x)+f^i(x)]\partial_i.
\end{equation*}

We will show that the sequence $X_t^n$ converges to the unique strong solution $X_t$ to the SDE (\ref{equa}) and that the moments of $DX_t^n$ are uniformly
bounded with respect to $n$ and $t$. This way we can use  Lemma 1.2.3 in \cite{Nualart06} to
derive the Malliavin differentiability of $X_t$ and show that $X_t \in \D^{\infty}$.
\begin{lem}
For each $t \in [0,T]$ and $p > 1$, the sequence $X_{t}^n$
converges to $X_{t}$ in $L^p(\Omega)$.
\end{lem}

\begin{proof}
To proceed, we prove the almost sure convergence of $X_{t}^n$ to $X_t$. Then by showing the uniform integrablility of $X_{t}^n$ we will conclude. \\
Let $X^{\tau_{n}}$ denote $X$ stopped at $\tau_{n}$.
By the choice of $\phi_{n}(.)$, it follows that $X_{t}^{\tau_{2n}}=X_{t}^{\tau_{n}}$ for all $t \leq \tau_{n}$.
So, for fixed $t \in [0,T]$, letting $n$ tend to $\infty$, we
have $lim_{n \rightarrow \infty}X_{t}^{n}=lim_{n \rightarrow \infty}X_{t}^{\tau_{n}}=X_{t}$ a.s.\\
Now, we are going to prove the uniform integrability of the sequence $X_{t}^n$. We will show that for every $p > 1$, there exists $c_{p}>0$ such that
\begin{equation}\label{ssup}
\sup_{n \geq 1}\sup_{0 \leq t \leq T}\E\Big[\vert X_{t}^n\vert^p\Big] \leq
c_{p}.
\end{equation}
By the definition of $\Lnt$, we have

$\begin{aligned}
\Lnt \vert X^n_t -x_0\vert^p & = p\vert X_t^n -x_0\vert^{p-2} \langle X_t^n-x_0 , b_n(X_t^n)+f(X_t^n) \rangle \\
&+ \frac{p}{2} \vert X_t^n -x_0\vert^{p-2}\vert \sigma(X_t^n) \vert^2 \\
  & + \frac{p(p-2)}{2} \vert X_t^n -x_0\vert^{p-4}\vert \langle X_t^n -x_0, \sigma(X_t^n) \rangle \vert^2 \\
  & =  p\vert X_t^n -x_0\vert^{p-2}\langle X_t^n -x_0, b_n(X_t^n)-b(x_0)\phi_{n}(X_{t}^{n}) \rangle \\
& + p\vert X_t^n -x_0\vert^{p-2}\langle X_t^n -x_0, b(x_0)\phi_{n}(X_{t}^{n}) +f(X_t^n) \rangle\\
  &+ \frac{p}{2} \vert X_t^n -x_0\vert^{p-2}\vert \sigma(X_t^n) \vert^2 \\
  &+ \frac{p(p-2)}{2} \vert X_t^n -x_0\vert^{p-4}\vert \langle X_t^n -x_0, \sigma(X_t^n) \rangle \vert^2 
 \end{aligned}$\\ 
By use of inequality $-ac \leq {a^2}/{2}+{c^2}/{2}$ for
$a=K$ and $c=\phi_{n}(X_{t}^{n}) $, and because $\phi_n(.) \leq 1$, by (\ref{monoton}) and (\ref{Bound}) one can find  positive constants $\alpha_p$, $\beta_p$ such that 
\begin{align} 
 \Lnt \vert X^n_t -x_0\vert^p  & \leq  -Kp\vert X_t^n -x_0\vert^{p} \phi_{n}(X_{t}^{n})\nonumber \\
&  + p\vert X_t^n -x_0\vert^{p-2}\langle X_t^n -x_0, b(x_0)\phi_{n}(X_{t}^{n}) +f(X_t^n) \rangle\Big]\nonumber\\
  & + \frac{p(p-1)}{2} \vert X_t^n-x_0 \vert^{p-2}\vert \sigma(X_t^n) \vert^2 \nonumber \\
  & \leq \frac{K^2+1}{2}p\vert X_t^n -x_0\vert^{p} \phi_{n}(X_{t}^{n}) \nonumber\\
  &+p\vert X_t^n -x_0\vert^{p-2}\Big[\vert X_t^n -x_0\vert^2+\frac{1}{2}\Big(\vert b(x_0)\vert^2+\vert f(X_t^n)\vert^2\Big) \Big]\nonumber\\
  & + \frac{p(p-1)}{2} \vert X_t^n -x_0\vert^{p-2}\vert \sigma(X_t^n) \vert^2 \nonumber\\
     & \leq \alpha_p \vert X_t^n -x_0\vert^{p} + \beta_p \vert X_t^n -x_0\vert^{p-2}. \label{Ln}
\end{align}
Using It\^o's formula, we have\\

$\begin{aligned}
\frac{d}{dt} \E\Big[\vert X^n_t -x_0\vert^p\Big] &= \E\Big[\Lnt(\vert X^n_t -x_0\vert^p)\Big] \\
&\leq \alpha_p \E
\Big[\vert X_t^n -x_0\vert^{p}\Big] + \beta_p \E\Big[\vert X_t^n -x_0\vert^{p-2}\Big].
\end{aligned}$\\
Setting $p=2$ and applying Gronwall's inequality, (\ref{ssup}) will be proved for $p=2$. By induction on $p$ and 
by the following inequality \\

$\begin{aligned}
\frac{d}{dt} \E\Big[\vert X^n_t -x_0\vert^p \Big] & = \E\Big[\Lnt(\vert X^n_t -x_0\vert^p)\Big] \\
& \leq \alpha_p \E
\Big[\vert X_t^n -x_0\vert^{p}\Big] + \beta_p \Big(\E\Big[\vert X_t^n -x_0\vert^{p-1}\Big]\Big)^{1-\frac{1}{p-1}},
\end{aligned}$\\
(\ref{ssup}) will be proved for every $p \geq 2$. \\
Now by  almost sure convergence of $X_t^n$ to 
$X_t$  and by inequality (\ref{ssup}) the proof of Lemma is completed.
\end{proof}

\section{Weak differentiability in the Wiener space}
\vskip 0.4 true cm
In this section, first  we use Lemma 1.2.3 in \cite{Nualart06} to
derive Malliavin differentiability of the solution to (\ref{equa}). Then we show that $X_t \in \D^{\infty}$. Notice that by Result \ref{solindinf}, the solutions to SDEs (\ref{equan}) are in $\D^\infty$.

\begin{lem}\label{firstderiv}
Assume that Hypothesis {\rm \ref{hypothesis}} holds, then  the unique strong
solution of SDE {\rm (\ref{equa})} is in $\D^{1,p}$ for every $p>1$. Moreover, for $r \leq t$ \\

$\begin{aligned}
\Drti  & =\sigma^i(X_{r})+\int_{r}^{t}[\nabla b^i(X_{s})+\nabla f^i(X_{s})].\Drs
ds\\
&+\int_{r}^{t}\nabla \sigma^i_l(X_{s}).\Drs dW^l_{s},
\end{aligned}$\\
and  for $r > t$, $\Drti =0$, where $\sigma_l(X_{s})$ is the l-th column of $\sigma(X_{s})$ and $u.C$ denotes the product $C^*u$ of  matrix $C^*$  and vector $u$.
\end{lem}

\begin{proof}
By Result \ref{solindinf} we know that for every $r \leq t$ and $1 \leq i\leq d$\\

$\begin{aligned}
\Dnrti & =\sigma^i(X^n_{r})+\int_{r}^{t}[\bnpsi+\fnpsi].\Dnrs
ds \\
&+\int_{r}^{t}\nabla \sigma^i_l(X_{s}^n).\Dnrs dW^l_{s},
\end{aligned}$\\
and for every $r > t$, $\Dnrti=0$.\\
Now by Lemma 1.2.3 in \cite{Nualart06}, it is sufficient to show that
\begin{equation}\label{karand}
\sup_{n \geq
1}\sup_{0 \leq t \leq T}\E\Big[\Vert \Dnt \Vert_{H}^p\Big] \leq c_{p}.
\end{equation}
To this end, note that for every $1 \leq i\leq d$ by It\^o's formula 
\begin{equation} \label{itoi}
\E\Big[\vert \Dnrti \vert^p\Big] =
      \E\Big[\vert \sigma ^i(X_{r}^n)\vert^p\Big]+\E\Big[\int_{r}^{t}\gnt \Big(\vert \Dnrsi \vert^p\Big)ds\Big]+\E\Big[M^n_t\Big] ,
\end{equation}
where
\\

$\begin{aligned}
\gnt\Big( \vert \Dnrsi \vert^p\Big) & = p \vert \Dnrsi \vert^{p-2}S_{i,s} \\
&+p\vert \Dnrsi \vert^{p-2}\langle \Dnrsi, \fnpsi.\Dnrs \rangle \\
   & + \frac{p}{2} \vert \Dnrsi \vert^{p-2}\vert \nabla \sigma^i_l(X_{s}^n).\Dnrs \vert^2 \\
      & +\frac{p(p-2)}{2} \vert \Dnrsi \vert^{p-4}\vert
 \langle \Dnrsi ,\nabla \sigma^i_l(X_{s}^n).\Dnrs \rangle \vert^2,\\
\end{aligned}$\\
\\
$$S_{i,s} :=\langle \Dnrsi, \bnpsi.\Dnrs \rangle,$$
and 
\begin{equation*} 
M^n_t:= \int_{r}^{t} p \vert \Dnrsi \vert^{p-2}
 \langle \Dnrsi ,\nabla \sigma^i_l(X_{s}^n).\Dnrs dW^l_s \rangle.
\end{equation*}
Notice that by Result \ref{solindinf},  $M_t^n$
is a local martingale and thus $\E[M_t^n]=0$.\\
Since $\sigma$ and $f$ have bounded derivatives, there exists some $\gamma > 0$ such that
\begin{align}
  \frac{p}{2} &\vert \Dnrsi \vert^{p-2} \vert \nabla \sigma^i_l(X_{s}^n).\Dnrs \vert^2  \nonumber\\
  & + \frac{p(p-2)}{2} \vert \Dnrsi \vert^{p-4}\vert
 \langle \Dnrsi ,\nabla \sigma^i_l(X_{s}^n).\Dnrs \rangle \vert^2   \nonumber\\
 & \leq   \gamma \frac{p(p-1)}{2} \vert \Dnrsi \vert^{p-2}\vert \Dnrs \vert^2,   \label{si}
\end{align}
and 
\begin{align}
  p\vert \Dnrsi \vert^{p-2}& \langle \Dnrsi, \fnpsi.\Dnrs \rangle \leq \nonumber\\
  &\frac{p}{2} \vert \Dnrsi \vert^{p}
 + \gamma \frac{p}{2} \vert \Dnrsi \vert^{p-2}\vert \Dnrs \vert^2.       \label{fi}
\end{align}
Using (\ref{ggg}) and   (\ref{supderbphi}), for  $0 \leq s \leq  T$ we have

\begin{align} 
   \sum_{i=1}^{d} S_{i,s} & =  \sum_{j=1}^{d} <\bnps \Dnrsj,\Dnrsj> \nonumber\\
   &=\sum_{j=1}^{d}\phi_{n}( X_{s}^{n})\langle \nabla b(X_{s}^{n}) \Dnrsj,\Dnrsj  \rangle    \nonumber\\
&+\sum_{j=1}^{d}\langle  \langle b(X_{s}^{n}),\nabla \phi_{n}(X_s^n) \rangle \Dnrsj,\Dnrsj \rangle  \nonumber\\
& \leq (-K \phi_{n}(X_{s}^{n})+M_1) \sum_{j=1}^{d}\vert \Dnrsj \vert^2  \leq M_1 \sum_{j=1}^{d}\vert \Dnrsj \vert^2                  \label{S1}
\end{align}
where $\Dnrsj$ is the $j$-th column of  $DX_s^n$. 
For every $Y=(Y^1, \cdots, Y^d) \in \R^d$ and $1 \leq i \leq d$ 
\begin{equation} \label{lpl}
\vert Y^i \vert^p \leq \vert Y \vert^p,
\end{equation}   
and
\begin{equation}\label{lp}
\vert Y \vert^p \leq 2^{\frac{p}{2}-1}\sum_{i}^{d} \vert Y^i \vert^p.
\end{equation} 
Thus substituting (\ref{S1}), (\ref{si}) and (\ref{fi}) in (\ref{itoi}) and taking summation on $i$ we derive: 

$\begin{aligned} 
\E\Big[\vert \Dnrt \vert^p\Big]  & \leq 2^{\frac{p}{2}-1}\sum_{i=1 }^{d} \E\Big[ \vert \Dnrti \vert^p \Big]  \\
&  \leq 2^{\frac{p}{2}-1}\sum_{i=1}^{d} \E\Big[\vert \sigma ^i(X_{r}^n)\vert^p\Big] \\
&+ 2^{\frac{p}{2}-1} pdM_1 \sum_{i=1}^{d}\int_{r}^{t}\E\Big[\vert \Dnrsi \vert^{p-2}\vert  \Dnrs \vert^2 \Big]ds \\
& +2^{\frac{p}{2}-1}\sum_{i=1}^{d}\int_{r}^{t}\E\Big[\frac{p}{2} \vert \Dnrsi \vert^{p}\Big]ds\\
& +2^{\frac{p}{2}-1}\sum_{i=1}^{d} \gamma \frac{p}{2}\int_{r}^{t}\E\Big[ \vert \Dnrsi \vert^{p-2}\vert \Dnrs \vert^2\Big]ds \\
 & +2^{\frac{p}{2}-1}\sum_{i=1}^{d} \gamma \frac{p(p-1)}{2}\int_{r}^{t}\E\Big[ \vert \Dnrsi \vert^{p-2}\vert \Dnrs \vert^2\Big]ds. \\
\end{aligned}$\\

Now we can find a constant $\alpha'_p > 0$ such that  

\begin{equation*}
\E\Big[\vert \Dnrt \vert^p\Big] \leq 2^{\frac{p}{2}-1}\sum_{i=1}^{d} \E\Big[\vert \sigma ^i(X_{r}^n)\vert^p\Big] + \alpha'_p \int_{r}^{t} \E\Big[\vert \Dnrs \vert^p \Big]ds .
\end{equation*}
Using  Gronwall's inequality, we have
\begin{equation*}
\E\Big[\vert \Dnrt \vert^p\Big] \leq 2^{\frac{p}{2}-1}\sum_{i=1}^{d} \E\Big[\vert \sigma ^i(X_{r}^n)\vert^p\Big] exp\{\alpha'_p T\}.
\end{equation*}
From which by the Lipschitz property of $\sigma$ and inequality (\ref{ssup}) the result follows. 
\end{proof}
Here we are going to prove higher order deffierentiability of $X_t$. For simplicity, we will only show the second order differentiability. For every real-valued function $f$ and random variables $F$ and $G$, set **$\triangle f(x) FG := \sum{i,j}\partial_{i}\partial_j f(x) F^i G^j$** and $\DD F= D^k_{\tau} D^{j}_r F$.

\begin{lem}\label{d2}
Assuming Hypothesis \ref{hypothesis}, the unique strong solution
of SDE {\rm (\ref{equa})} is in $\D^{2,p}$, for every $p>1$, and

$\begin{aligned}
\DD X^i_t   &=A^{ij}_{\tau, r} \\
&+\int_{\tau \vee r}^{t}\Big[\langle \nabla \sigma^i_l(X_{s}), \DD X_s \rangle +\triangle \sigma_l^i(X_{s})\Dtauk X_{s}\Drsj \Big] dW^l_{s} \\
  & + \int_{\tau \vee r}^{t} \langle \bpsi+\fpsi , \DD X_s \rangle ds\\
  & +\int_{\tau \vee r}^{t}\Big[\triangle b^i(X_{s})+\triangle f^i(X_{s})\Big]\Dtauk X_{s}\Drsj ds,\\
\end{aligned}$\\
where
\begin{equation*}
A^{ij}_{\tau, r}=\langle \nabla \sigma_j^i(X_{r}), \Dtauk X_{r} \rangle+\sum_{l=1}^{d}\langle \nabla \sigma^i_l(X_{\tau}), D_{r}^jX_{\tau} \rangle,
\end{equation*}
and $D_{\tau}X_{r}=0$ for $\tau > r$, and 
$D_{r}X_{\tau}=0$ for $\tau < r$.
\end{lem}

\begin{proof}
Since $X^n_{t}\in\D^\infty$,  by Result \ref{solindinf} for $\tau_0 := \tau \vee r$ we have

$\begin{aligned}
\DD (X_t^n)^i & =A_{n, \tau, r}^{ij} \\
&+\int_{\tau_0}^{t}\Big[\langle \nabla \sigma^i_l(X^n_{s}),\DD X_s^n \rangle +\triangle \sigma_l^i(X^n_{s})\Dtauk X^n_{s}\Dnrsj \Big] dW^l_{s} \\
& + \int_{\tau_0}^{t} \langle \bnpsi+\fpsi, \DD X_s^n \rangle ds \\
& +\int_{\tau_0}^{t}\Big[\triangle b_n^i(X^n_{s})+\triangle f^i(X^n_{s})\Big]\Dtauk X^n_{s} \Dnrsj ds,                 
\end{aligned}$\\
where
\begin{equation*}
A^{ij}_{n, \tau, r}=\langle\nabla \sigma_j^i(X^n_{r}), \Dtauk X^n_{r}\rangle+\sum_{l=1}^{d}\langle \nabla \sigma^i_l(X^n_{\tau}), D_{r}^j X^n_{\tau}\rangle,
\end{equation*}
and $D_{\tau}X^n_{r}=0$ for $\tau > r$. Similarly we have
$D_{r}X^n_{\tau}=0$ for $\tau < r$.
 By Lemma 1.2.3 in \cite{Nualart06}, it remains only to find some $c_2 > 0$ such that 
\begin{equation}\label{karandd}
\sup_n \E\Big[ \Vert D^{j,k}X_t^n\Vert^{p}_{H \otimes H} \Big] < c_2. 
\end{equation}
By It\^o's formula, for every $1 \leq i\leq d$ we have
\begin{equation}\label{itodd}
\E\Big[\vert \DD (X_t^n)^i\vert^p\Big] =
      \E\Big[\vert A^{ij}_{n, \tau, r} \vert^p\Big]+\E \Big[ \int_{\tau}^{t}\gntij \Big(\vert  \DD (X_s^n)^i \vert^p\Big)ds\Big]+\E\Big[M^{ij}_n(t))\Big],
\end{equation}
where
\\
$\begin{aligned}
\gntij \Big( \vert \DD (X_s^n)^i\vert^p\Big) & = p\vert  \DD (X_s^n)^i \vert^{p-2}I_1 + \frac{p}{2} \vert  \DD (X_s^n)^i \vert^{p-2}\sum_{l=1}^{d} I_2(l) \nonumber\\
& +\frac{p(p-2)}{2} \vert  \DD (X_s^n)^i \vert^{p-4}I_3 , 
\end{aligned}$\\
in which \\
$
\begin{aligned}
I_1 & := \DD (X_s^n)^i \Big(\langle \bnpsi+\fnpsi, \DD X_s^n \rangle \Big) \\
& +\Big[\triangle b_n^i(X^n_{s})+\triangle f^i(X^n_{s})\Big]\Dtauk X^n_{s}\Dnrsj,
\end{aligned}$\\
\begin{center}
$I_2(l):=\Big[\vert  \triangle \sigma_l^i(X^n_{s})\Dtauk X^n_{s}\Dnrsj   \vert +\vert \langle \nabla \sigma^i_l(X_{s}^n), \DD X_s^n \rangle\vert\Big]^2 ,$
\end{center}
\begin{center}
$I_3:=\vert
\DD (X_s^n)^i\Big(\triangle \sigma_l^i(X^n_{s})\Dtauk X^n_{s}\Dnrsj +\langle \nabla \sigma^i_l(X_{s}^n), \DD X_s^n\rangle  \Big)\vert^2,$
\end{center}
and    
\begin{equation*} 
M^{ij}_n(t) := \int_{r}^{t} p \vert \DD (X_s^n)^i \vert^{p-2}
 \langle \DD (X_s^n)^i , I_2(l) dW^l_s \rangle.  
\end{equation*}
Notice that by Result \ref{solindinf},  $M^{ij}_n(t)$
is a local martingale and thus $\E[M^{ij}_n(t)]=0$.\\
Now, we are going to find  appropriate upper bounds for $I_1, I_2(l)$ and  $I_3$. As $\sigma$ has bounded derivatives, we can find some $\gamma'_1 > 0$ such that
\begin{align}
& \frac{p}{2} \vert \DD (X_s^n)^i  \vert^{p-2} \sum_{l=1}^{d}I_2(l) +\frac{p(p-2)}{2} \vert \DD X_s^n \vert^{p-4}I_3  \leq  \nonumber\\
    & \gamma'_1 \frac{p(p-1)}{2} \Big( \vert \DD (X_s^n)^i \vert^{p-2}\vert \DD X_s^n \vert^2  +  \vert \DD (X_s^n)^i  \vert^{p-2}\vert \Dnrsj \vert^2 \vert \Dtauk X_s^n \vert^2 \Big). \label{I2}
\end{align}
Also by the boundedness of $f$ and the derivatives of $\sigma$, the polynomial growth of the derivatives of $b$ and (\ref{roshdbn}), there exist some $\gamma'_2 > 0$ and $q >0$ such that

\begin{align}
  p\vert \DD (X_s^n)^i \vert^{p-2} I_1  & = p\vert \DD (X_s^n)^i \vert^{p-2} J_1 +p\vert \DD (X_s^n)^i \vert^{p-2}J_2   \nonumber\\
  &  +p\vert \DD (X_s^n)^i \vert^{p-2}\DD (X_s^n)^i \langle \fnpsi, \DD X_s^n \rangle \nonumber\\ 
 & \leq p\vert \DD (X_s^n)^i \vert^{p-2} J_1 \nonumber\\
 &+ \gamma'_2 p \vert \DD (X_s^n)^i \vert^{p-2}\vert \Dtauk X_s^n \vert^{2}\vert \Dnrsj \vert^2(1+\vert X_s^n \vert^{p_{_{2}}})^2 \nonumber\\
 & +p \gamma'_2 \vert \DD (X_s^n)^i  \vert^{p}+p\gamma'_2 \vert \DD (X_s^n)^i  \vert^{p-2}  \vert \DD (X_s^n)^i  \vert^{2},\label{I1}
\end{align}
where
\begin{center}
$ J_1:=  \DD (X_s^n)^i \langle \bnpsi ,\DD X_s^n \rangle ,$
\end{center}
and
\begin{center}
$ J_2:= \DD (X_s^n)^i \Big( \Big[\triangle b_n^i(X^n_{s})+\triangle f^i(X^n_{s})\Big] \Dtauk X^n_{s} \Dnrsj\Big)$
\end{center}
By (\ref{ggg}) and   (\ref{supderbphi}), for every $0 \leq s \leq  T$ we have

\begin{align} 
   \sum_{i=1}^{d} J_1 & =  \langle \bnps \DD X_s^n,\DD X_s^n \rangle =\phi_{n}(X_{s}^{n})\langle \nabla b(X_{s}^{n}) \DD X_s^n,\DD X_s^n  \rangle    \nonumber\\
&+\langle  \langle b(X_{s}^{n}),\nabla \phi_{n}(X_s^n) \rangle \DD X_s^n,\DD X_s^n \rangle  \nonumber\\
& \leq (-K \phi_{n}(X_{s}^{n})+M_1) \vert \DD X_s^n \vert^2  \leq M_1 \vert\DD X_s^n \vert^2.                  \label{J1}
\end{align}
Now, substitute (\ref{I1}) and (\ref{I2}) in (\ref{itodd}), sum up on $i$ and then use (\ref{J1}) and (\ref{lpl}) to derive 
\begin{align} 
\sum_{i=1}^{d}& \E\Big[\vert  \DD (X_t^n)^i\vert^p\Big]  = \nonumber\\
&\sum_{i=1}^{d} \E\Big[\vert A^{ij}_{n, \tau, r} \vert^p\Big]  + p(M_1+2d\gamma'_2 + d\gamma'_1 \frac{p(p-1)}{2}) \int_{\tau_0}^{t}\E\Big[\vert \DD X_s^n \vert^{p}\Big]ds\nonumber\\
 & +\sum_{i=1}^{d} \gamma'_2 p\int_{\tau_0}^{t}\E\Big[ \vert \DD (X_s^n)^i \vert^{p-2}\vert \Dtauk X_s^n \vert^{2}\vert \Dnrsj \vert^2(1+\vert X_s^n \vert^{p_{_{2}}})^2  \Big] ds\nonumber\\
  & +\sum_{i=1}^{d} \gamma'_1 \frac{p(p-1)}{2}\int_{\tau_0}^{t}\E\Big[ \vert \DD (X_s^n)^i  \vert^{p-2}\vert \Dnrsj \vert^2 \vert \Dtauk X_s^n \vert^2 \Big]ds. \label{main}
\end{align}\\
To bound the terms in the right hand side of the above inequality, we need the following version of the Young's inequality.
For $p \geq 2$ and for all $a,c$ and $\triangle_1
>0$ we have:
\begin{equation}\label{young}
a^{p-2}c^2 \leq \triangle_1^2
\frac{p-2}{p}a^p+\frac{2}{p\triangle_1^{p-2}}c^p.
\end{equation}
Using (\ref{young}) with $\triangle_1=1$ and $a=\vert \DD (X_s^n)^i  \vert$ we find some bounds for the last four terms in (\ref{main}) which depend only on $\int_{\tau_0}^{t}\E\Big[\vert \DD X_s^n  \vert^p\Big]ds$ and some terms which could be bounded by a constant. So for the last term in (\ref{main}) we have  

$\begin{aligned} 
 & \sum_{i=1}^{d} \gamma'_1 \frac{p(p-1)}{2}\int_{\tau_0}^{t}\E\Big[ \vert \DD (X_s^n)^i  \vert^{p-2}\vert \Dnrsj \vert^2 \vert \Dtauk X_s^n \vert^2 \Big]ds \leq \\
 &  d\gamma'_1 \int_{\tau_0}^{t}\Big(\frac{(p-1)(p-2)}{2}\E\Big[\vert \DD X_s^n  \vert^p\Big]+(p-1)\E\Big[\vert \Dnrsj \vert^p \vert \Dtauk X_s^n \vert^p \Big]\Big)ds.    
\end{aligned}$\\  
and for the third summand in (\ref{main}) we have 

$\begin{aligned} 
 & \sum_{i=1}^{d} \gamma'_2 p\int_{\tau_0}^{t}\E\Big[ \vert \DD (X_s^n)^i  \vert^{p-2}\vert \Dtauk X_s^n \vert^{2}\vert \Dnrsj \vert^2(1+\vert X_s^n \vert^{p_{_{2}}})^2 \Big]ds \leq  \\
 & d\gamma'_2 \int_{\tau_0}^{t}\Big((p-2) \E\Big[\vert \DD X_s^n \vert^p\Big]+2\E\Big[\vert \Dtauk X_s^n \vert^{p}\vert \Dnrsj \vert^p(1+\vert X_s^n \vert^{p_{_{2}}})^p\Big]\Big)ds.    
\end{aligned}$\\ 
Substituting these bounds in the right hand side of (\ref{itodd}) and using (\ref{ssup}), (\ref{karand}) and (\ref{lp}), we can find
some positive constants $c_1(p)$ and $c_2(p)$ such that  

\begin{equation*}
\E\Big[\vert  \DD X_s^n \vert^p\Big] \leq 2^{\frac{p}{2}-1} \sum_{i=1}^{d} \E\Big[\vert A_{n, \tau, r}^{ij} \vert^p\Big]+c_2(p) +  c_1(p)\int_{\tau_0}^{t} \E\Big[\vert  \DD X_s^n \vert^p \Big]ds. 
\end{equation*}
Now, from (\ref{karand}), (\ref{ssup}) and  the definition of $ A_{n, \tau, r}^{ij}$ (in which we have used the boundedness of the derivatives of $\sigma$),  Gronwall's inequality gives us  (\ref{karandd}). 
\end{proof}
In the same way, one can easily show that for every multiindex $\alpha$
\begin{equation}\label{supalld}
\sup_{n} \E(\rVert D^\alpha X^{n}_{t} \rVert^p_{H^{\otimes
\alpha}}) < \infty
\end{equation}
and then by Lemma 1.2.3 in \cite{Nualart06} deduce the following theorem.
\begin{thm}
The SDE (\ref{equa}) has a unique strong solution in $\D^{\infty}$.
\end{thm}
\appendix
\section{Constructing the approximating functions for the drift}
\vskip 0.4 true cm
Here we construct the functions $b_n$ mentioned in section 2. This construction is motivated by Berhanu in \cite[Theorem 2.9.]{Berhanu}.  Assume that $U \subset V$ are two open sets in $\R^d$ with distance $a>0$. For $0 \leq \epsilon \leq a$, define 
$U_{\epsilon}=\{x ; d(x, U) < \epsilon \}$. Then $U_\epsilon = \bigcup_{x \in U} B_\epsilon(x)$ and $U \subseteq U_\epsilon \subseteq V$. Fix $\epsilon$ such that 
$0 < 2\epsilon \leq a$ and let $h^\epsilon(x) $ be the characteristic function of $U_\epsilon$. For $\psi \in C_0^\infty(\R^d)$ with $supp \psi \subseteq  B_1(0)$ and
$\int \psi(x) dx = 1$, set $\psi_\epsilon(x) =
\frac{1}{\epsilon^{d}}\psi(\frac{x}{\epsilon})$. Now consider  the construction
function $$\psi_\epsilon \star h^\epsilon(x)=\int_{\R^d} \psi_\epsilon(y) h^\epsilon(x-y)dy$$ for $0 < 2\epsilon < d$. Since $supp \psi_\epsilon \subseteq B_\epsilon(0)$, $\psi_\epsilon \star h^\epsilon=1$ on $U$ and 
$\psi_\epsilon \star h^\epsilon=0$ outside $U_{2\epsilon}$. Note that 
for each multiindex $\alpha$,

\begin{align}
\partial_\alpha (\psi_\epsilon \star h^\epsilon) (x) & = \int \partial_\alpha ( \psi_\epsilon(y)) h^\epsilon(x-y) dy  = \frac{1}{\epsilon^{d+\vert \alpha \vert}}\int (\partial_\alpha \psi)(\frac{y}{\epsilon}) h^\epsilon(x-y) dy \nonumber \\
& =  \frac{1}{\epsilon^{\vert \alpha \vert}}\int (\partial_\alpha \psi)(z) h^\epsilon(x-\epsilon z) dz \leq \parallel \psi\parallel_\infty \frac{1}{\epsilon^{\vert \alpha \vert}}   \label{consphin}
\end{align}
Now, let $n \geq 1$ and set $U=B_{n^\xi}(0)$, $V=B_{2n^\xi}(0)$ and $\epsilon=n^\xi $. Then the  functions $\phi_n(x) := \psi_{\epsilon} \star h^\epsilon$ satisfy  $\phi_n(x)=1$ on $U$ and $\phi_n(x)=0$ outside $V$. Since $supp \phi_n(x) \subseteq B_{2n^\xi}(0)$, by (\ref{consphin}) and (\ref{roshdb}) for each multiindex $\alpha$ with $\vert \alpha \vert=c \geq 1$, we have \\

$\begin{aligned}
\vert b(x) \partial_\alpha \phi_n(x) \vert & \leq \vert b(x)\chi_{\vert x \vert \leq 2n^\xi} \vert \parallel \psi\parallel_\infty \frac{1}{n^{\xi \vert \alpha \vert}} \\
& \leq \gamma_c (1+ 2^{\xi} n^\xi)\parallel \psi\parallel_\infty \frac{1}{n^{\xi \vert \alpha \vert}} \leq 2^{\xi+1}\gamma_c \parallel \psi\parallel_\infty,
\end{aligned}$\\
and 
\begin{equation*}
\vert \partial_\alpha \phi_n(x) \vert \leq \parallel \psi\parallel_\infty. 
\end{equation*}
\\
\vskip 0.4 true cm
\begin{center}{\textbf{Acknowledgments}}
\end{center}
The authors wish to thank professor Bijan Z. Zangeneh  for insightful suggestions leading to a much improved version of the paper.  \\
\vskip 0.4 true cm

\bigskip
\bigskip


{\footnotesize \pn{\bf First Author}\; \\ {Department of Applied Mathematics, Faculty of Mathematical Sciences}, {Tarbiat Modares University,
P.O. Box 14115-134,}  {Tehran, Iran}\\
{\tt Email: tahmasebi@modares.ac.ir}\\

{\footnotesize \pn{\bf Second Author}\; \\ {Graduate School of Management and Economics}, {Sharif University of Technology,
P.O. Box 11155-9415,} {Tehran, Iran}\\
{\tt Email: zamani@sharif.ir}\\
\end{document}